\documentclass[12pt,reqno,a4paper]{amsart}

\usepackage{fullpage}
\usepackage{setspace}
\usepackage{datetime}
\usepackage{verbatim}
\usepackage{enumerate}
\usepackage{color}
\usepackage{url}
\usepackage{dsfont}
\usepackage[dvipsnames]{xcolor}
\usepackage{amsmath}

\usepackage[shortlabels]{enumitem} 

\usepackage{lineno} 

\newtheorem{theorem}             {Theorem}[section]
\newtheorem{lemma}     [theorem] {Lemma}

\newtheorem{definition}[theorem] {Definition}   
\newtheorem{proposition}[theorem] {Proposition}

\newtheorem{remark}[theorem] {Remark}

\newtheorem{proto-theo}[theorem] {Proto-Theorem}   

\usepackage[colorlinks=false]{hyperref}

\numberwithin{equation}{section}

\newcommand{\rarrow}{
  \xrightarrow[{\raisebox{.5mm}[1mm][0mm]
    {$\scriptstyle \rm $}}]{\raisebox{0.0mm}[0mm]
    {$\scriptstyle r$}}}

\newcommand{\tarrow}{
  \xrightarrow[{\raisebox{.5mm}[1mm][0mm]
    {$\scriptstyle \rm $}}]{\raisebox{0.0mm}[0mm]
    {$\scriptstyle  \rm 2 $}}}

\newcommand{\canarrow}{
  \xrightarrow[{\raisebox{.5mm}[1mm][0mm]
    {$\scriptstyle \rm $}}]{\raisebox{0.0mm}[0mm]
    {$\scriptstyle \rm can$}}}

\newcommand{\canarrowL}{
  \xrightarrow[{\raisebox{.5mm}[1mm][0mm]
    {$\scriptstyle\cL$}}]{\raisebox{0.0mm}[0mm]
    {$\scriptstyle \rm can$}}}

\let\:=\colon

\newcommand{\cC}{\mathcal{C}}

\newcommand{\cF}{\mathcal{F}}

\newcommand{\cH}{\mathcal{H}}
\newcommand{\cI}{\mathcal{I}}

\newcommand{\cL}{\mathcal{L}}

\newcommand{\cP}{\mathcal{P}}

\newcommand{\cS}{\mathcal{S}}
\newcommand{\cT}{\mathcal{T}}

\def\bbn{{\mathbb N}}
\def\NN{{\mathbb N}}

\def\bfg{{\mathbf G}}

\def\calf{{\mathcal F}}

\def\calh{{\mathcal H}}
\def\cali{{\mathcal I}}

\def\call{{\mathcal L}}

\def\cals{{\mathcal S}}

\let\cH\calh
\let\cP\calp

\let\cF\calf
\let\cI\cali
\let\cS\cals
\let\cC\calc
\let\cL\call

\let\epsilon\varepsilon
\let\eps\epsilon
\let\subset\subseteq

\usepackage{centernot}

\makeatletter
\newcommand{\definetitlefootnote}[1]{%
  \newcommand\addtitlefootnote{%
    \makebox[0pt][l]{$^{*}$}%
    \footnote{\protect\@titlefootnotetext}
  }%
  \newcommand\@titlefootnotetext{\spaceskip=\z@skip $^{*}$#1}%
}
\makeatother

\begin{document}

\pagestyle{plain}
\thispagestyle{empty}
\footskip=30pt
\shortdate
\settimeformat{ampmtime}
\onehalfspacing


\definetitlefootnote{ An extended abstract introducing our result here (focusing just on the case of graphs) appeared in the proceedings of the XII Latin-American Algorithms, Graphs and Optimization Symposium (LAGOS 2023)~\cite{AKMMconference}}

\title[A canonical Ramsey theorem  with list constraints]%
{A canonical Ramsey theorem  with list constraints in random (hyper-)graphs\addtitlefootnote}

\author[José~D.~Alvarado]{José~D.~Alvarado}
\address{Faculty of Mathematics and Physics, University of Ljubljana, Jadranska 19, Ljubljana, Slovenia}
\email{jose.alvarado@fmf.uni-lj.si}

\author[Yoshiharu Kohayakawa]{Yoshiharu Kohayakawa}
\address{Instituto de Matem\'atica e Estat\'{\i}stica \\Universidade de 
	S\~ao Paulo, Rua do Mat\~ao~1010, 05508--090~S\~ao Paulo, Brazil}
\email{\{yoshi\,|\,mota\}@ime.usp.br}

\author[Patrick Morris]{Patrick Morris}
\address{Departament de Matem\`atiques, Universitat Polit\`ecnica de Catalunya (UPC), Carrer de Pau Gargallo 14, 08028  Barcelona, Spain.}
\email{pmorrismaths@gmail.com}
\author[Guilherme~O.~Mota]{Guilherme~O.~Mota}

\thanks{J. D. Alvarado was partially supported by FAPESP
  (2020/10796-0).  %
  Y. Kohayakawa was partially supported by FAPESP (2023/03167-5) and
  CNPq (407970/2023-1, 420838/2025-2, 315258/2023-3).  %
  P. Morris was partially supported by the Deutsche
  Forschungsgemeinschaft (DFG, German Research Foundation) Walter
  Benjamin program - project number 504502205.  %
  G. O. Mota was partially supported by CNPq (306620/2020-0,
  406248/2021-4) and FAPESP (2018/04876-1, 2019/13364-7).  %
  This study was financed in part by CAPES, Coordenação de
  Aperfeiçoamento de Pessoal de Nível Superior, Brazil, Finance
  Code~001.  FAPESP is the S\~ao Paulo Research Foundation.  CNPq is
  the National Council for Scientific and Technological Development of
  Brazil.}

\date{\today, \currenttime}

\begin{abstract}
  The celebrated canonical Ramsey theorem of Erd\H{o}s and Rado implies
  that for a given $k$-uniform hypergraph  (or $k$-graph) $H$, if $n$ is sufficiently large then any
  colouring of the edges of the complete $k$-graph $K^{(k)}_n$ gives rise to copies of $H$ that
  exhibit certain  colour patterns. We are interested in sparse random versions of this
  result and the thresholds at which the random $k$-graph $\bfg^{(k)}(n,p)$
  inherits the canonical Ramsey properties of $K^{(k)}_n$. Our main result
  here pins down this threshold when we focus on colourings that are
  constrained by some prefixed lists. This result is applied in an
  accompanying work of the authors on the threshold for the canonical
  Ramsey property (with no list constraints) in the case that $H$ is
  a (2-uniform) even cycle. 

\end{abstract}
	
\maketitle

\section{Introduction}
\label{sec:Introduction}
We begin by focusing on ($2$-uniform) graphs. For $r\in \mathbb{N}$ and graphs~$G$ and~$H$, we say~$G$ has the
\emph{$r$-Ramsey property with respect to~$H$}, denoted $G\rarrow
H$, if every colouring of the edges of $G$ with $r$ colours results in
a \textit{monochromatic} copy of~$H$, that is, a copy of~$H$ with all
its edges in the same colour. The classical theorem of
Ramsey~\cite{R30}, from which the term \textit{Ramsey theory} stems,
states that if $n$ is large enough in terms of $r$ and $H$, then
$K_n\rarrow H$. In a highly influential work, Erd\H{o}s and
Rado~\cite{ER50} explored which ``colour patterns'' are guaranteed when one
colours a graph with no restriction on the number of colours. Clearly,
monochromatic copies of $H$ are no longer guaranteed as one can colour
each edge of $K_n$ with a unique colour. In such a case, any copy of
$H$ in $K_n$ is said to be \textit{rainbow}. If $H$ contains a cycle,
it is also not the case that every colouring of $K_n$ induces either a
monochromatic or rainbow copy of $H$. Indeed, one can associate a
unique colour $c(i)$ to each vertex $i$ in $[n]=V(K_n)$ and colour
each edge $ij\in E(K_n)$ by $c(\min\{i,j\})$. Then any copy of $H$ in
$K_n$ is neither monochromatic nor rainbow but is coloured
\textit{lexicographically}.

\begin{definition}[Lexicographic colouring] \label{def:lex_copy}
  Let $H$ be a graph, $\sigma$ an ordering of $V(H)$ and
  $\chi:E(H)\rightarrow \bbn$ an edge colouring of $H$. We say that
  the pair $(H,\chi)$ is \emph{lexicographic with respect to~$\sigma$}
  if there exists an injective assignment of colours
  $\phi:V(H)\rightarrow \bbn$ such that for every edge $e=uv\in E(H)$
  with $u<_\sigma v$, we have that $\chi(e)=\phi(u)$.  If~$\chi$ is
  clear from the context, we simply say that~$H$ is lexicographic with
  respect to~$\sigma$.
\end{definition}

The celebrated \textit{canonical Ramsey theorem} of Erd\H{o}s and Rado
\cite{ER50} implies that if $n$ is large enough in terms of $m\in
\mathbb{N}$, then any colouring of $K_n$ results in a copy of $K_m$
that is either monochromatic, rainbow or lexicographic. This theorem
serves as a beautiful example of the popular Ramsey theory maxim that
there is an inevitable order amongst chaos. Applying the canonical
Ramsey theorem with $m=v(H)$ implies the existence of copies of~$H$ with certain colourings. The following definition captures this behaviour.

\begin{definition}[The canonical Ramsey property] \label{def:canonical
    ramsey property} Given a graph $H$ and an ordering~$\sigma$
  of~$V(H)$, an edge-coloured copy of $H$ is \emph{canonical with respect to
    $\sigma$} if it is monochromatic, rainbow or lexicographic with
  respect to $\sigma$.  A graph $G$ has the \emph{$H$-canonical Ramsey
    property}, denoted $G\canarrow H$, if for every edge colouring
  $\chi:E(G)\rightarrow \bbn$ and every ordering~$\sigma$ of~$V(H)$,
  there is a copy of $H$ which is canonical with respect to~$\sigma$.
   \end{definition}
 
   Note that in the case that there are neither monochromatic nor
   rainbow copies of $H$, our definition requires copies of $H$ with
   \textit{all} possible lexicographic colourings (that is, for
   every~$\sigma$, we require a lexicographic colouring with respect
   to~$\sigma$). 
   This notion of a canonical Ramsey property is therefore as strong as possible: there is no other \textit{colour pattern}\footnote{A colour pattern of a graph~$H$
     is a partition of its edge set.} that one can guarantee when avoiding monochromatic and rainbow copies of $H$. Indeed,
   as observed by Jamison and West~\cite{JW04}, the
   canonical Ramsey theorem implies that for any $H$, a set~$\cC$ of
 colour patterns 
   of~$H$ is \textit{unavoidable}
   in all colourings of large enough cliques\footnote{That is, large
     enough cliques are such that any colouring of their edges admits
     a copy of~$H$ whose edges are partitioned according to one of the
   given patterns.} 
   if and only if $\cC$
   contains the monochromatic pattern, the rainbow pattern and at
   least one lexicographic pattern.  Alternative definitions of canonical Ramsey
   properties are discussed in Section~\ref{sec:alts}.

   \subsection{Sparse Ramsey theory and random graphs}
   Returning to the setting of colourings with a bounded number of
   colours, a prominent theme in Ramsey theory has been to explore the
   existence of \textit{sparse} graphs $G$ that are $r$-Ramsey with
   respect to~$H$; see for example \cite{NR84} and the references
   therein. One famous example is the work of Frankl and
   R\"odl~\cite{FR86}, who used a random graph to construct a
   $K_4$-free graph $G$ such that $G\tarrow K_3$. This prompted
   {\L}uczak, Ruci\'nski and Voigt \cite{LRV92} to initiate the study
   of thresholds for Ramsey properties in random graphs, which has
   since become a prominent theme in probabilistic combinatorics. It
   turns out that the threshold for $\bfg(n,p)$ having the Ramsey
   property with respect to a graph $H$ is governed by the following
   parameter of $H$.

\begin{definition} \label{def:m2 density} Given a graph $H$ with at
  least two edges, the \emph{maximum $2$-density} of $H$ is defined by
  \[
    m_2(H):=\max\left\{\frac{e(F)-1}{v(F)-2}:F\subseteq H,\, v(F)>2
    \right\}.
  \]
\end{definition}

In a seminal series of papers, R\"odl and Ruci\'nski
\cite{RR93,RR94,RR95} established the threshold for the Ramsey
property when the random graph is coloured with a bounded number of
colours.  Here and throughout, we say that a function
$\hat p = \hat p(n)$ is the \emph{threshold} for a monotone increasing
graph property $\cP$ if
\[
  \lim_{n\to\infty}\Pr(\bfg(n,p)\text{ satisfies }\cP) =
  \begin{cases}0&\text{if }p =o( \hat p), \\
    1&\text{if }p =\omega( \hat p).
  \end{cases}
\]
We refer to $\hat p$ as \textit{the} threshold for $\cP$, although it
is only defined up to order of magnitude.  We now state an abridged
form of the Rödl--Ruci\'nski theorem.

 \begin{theorem}[Rödl--Ruci\'nski~\cite{RR93,RR94,RR95}]
    \label{RR-original}
    Let $r\geq 2$ be an integer and~$H$ be a graph  that is not a star forest.  Then $n^{-1/m_2(H)}$ is the
    threshold for the property that ${\bfg(n,p)\rarrow H}$.
\end{theorem}

\subsection{Canonical Ramsey properties of random
  graphs}
\label{sec:lists}

The motivation for the current work is to establish the threshold for
the canonical Ramsey property with respect to a given graph~$H$. Note
that for any $H$ as in Theorem~\ref{RR-original} that is not a
triangle, the threshold for $\bfg(n,p)\canarrow H$ is at least
$n^{-1/m_2(H)}$. Indeed for any such $H$, there is an
ordering~$\sigma$ of its vertices such that the lexicographic
colouring of $H$ with respect to $\sigma$ uses at least $3$
colours. Using Theorem~\ref{RR-original}, we have that when
$p=o\left(n^{-1/m_2(H)}\right)$, asymptotically almost surely (a.a.s.\
from now on) there is a $2$-colouring of $\bfg(n,p)$ that avoids
monochromatic copies of~$H$. Moreover, such a colouring avoids rainbow
and lexicographic copies of~$H$ with respect to $\sigma$, simply
because there are not enough colours available for such colour
patterns. This shows that $\bfg(n,p)$ does not have the canonical
Ramsey property for $H$ for such~$p$.
 
We believe that this lower bound is in fact the correct threshold for
the canonical Ramsey property and that when
$p=\omega\left(n^{-1/m_2(H)}\right)$, a.a.s.\ $\bfg(n,p)\canarrow H$.
Here, we provide evidence for this by focusing on colourings that
are constrained to be compatible with a given list assignment.

\begin{definition}[Colourings with list constraints] \label{def:list
    colourings} Let $1\leq r \in \bbn$ and $\cL:E(K_n)\rightarrow
  \bbn^r$ be an assignment of lists of colours to the edges of $K_n$
  (note that we allow lists to have repeated colours). We say that a
  colouring $\chi:E(G)\rightarrow \bbn$ of an $n$-vertex graph
  $G\subseteq K_n$, is \emph{compatible} with $\cL$ if for all $e\in
  E(G)$, we have that $\chi(e)\in \cL(e)$.
\end{definition}

Our main theorem shows that for any assignment $\cL$ of bounded lists
to the edges of~$K_n$, the threshold for the canonical Ramsey property
with respect to $H$ when considering colourings that are compatible
with~$\cL$ is at most~$n^{-1/m_2(H)}$.  Let us write $G\canarrowL H$
if any edge colouring~$\chi$ of~$G$ that is compatible with~$\cL$
contains a canonical copy of~$H$ with respect to~$\sigma$ for all
orderings~$\sigma$ of~$V(H)$.

\begin{theorem}\label{thm:main}
  Let~$H$ be a graph with at least two edges.  Let $1 \leq r\in \bbn$
  and let $\cL:E(K_n)\rightarrow \bbn^r$ be a list assignment of
  colours.  If $\bfg\sim\bfg(n,p)$ with
  $p = \omega\left(n^{-1/m_2(H)} \right)$, then $\bfg\canarrowL H$
  a.a.s.
\end{theorem}   

Note that for many assignments of lists $\cL$ this theorem establishes
the threshold for the canonical Ramsey property restricted to
colourings compatible with $\cL$. Indeed, by the same reasoning
discussed above, this is the case whenever there are $2$ colours that
feature on all lists.  We will deduce Theorem~\ref{thm:main} from a
more general result, namely Theorem~\ref{thm:main_list_general}, which
deals with graphs of the form $\Gamma \cap \bfg(n,p)$ where $\Gamma$
is a ``locally dense graph'' (see Section~\ref{sec:localdense} for the
relevant definitions).

\subsection{An application for even cycles}
\label{sec:even}
We believe that Theorem~\ref{thm:main} provides a natural stepping
stone towards establishing $n^{-1/m_2(H)}$ as the threshold for the
canonical Ramsey property in random graphs for all graphs~$H$ that are
not forests and are different from the triangle.  In fact, the theorem
arose naturally in the work of the authors proving a $1$-statement for
the canonical Ramsey property with respect to even cycles.  Indeed,
Theorem~\ref{thm:main} (or rather its stronger version,
Theorem~\ref{thm:main_list_general}) is a key component of the proof
of the following theorem, which is given in an accompanying
paper~\cite{AKMM-2}.

\begin{theorem}\label{thm:canon RG even-cycles}
  Let $k\geq 2$ be an integer. 
  If $p= \omega\left(
  n^{-(2k-2)/(2k-1)} \, \log{n}\right)$,  then a.a.s.\
  \[
    \bfg(n,p) \canarrow C_{2k}.
  \]
\end{theorem}
As we mentioned before, $n^{-1/m_2(H)}$ is a lower bound for the
threshold for the canonical Ramsey property with respect to even
cycles. Thus, since~$m_2(C_{2k})=(2k-1)/(2k-2)$,
Theorem~\ref{thm:canon RG even-cycles} establishes the threshold for
the canonical Ramsey property with respect to even cycles, up to the
$\log$ factor.

\subsection{Hypergraphs} \label{sec:hyper}

With our methods, we can also prove an analogue of Theorem \ref{thm:main} for hypergraphs. In order to introduce this, we need to adapt our definitions appropriately. For $k\geq 2$, we refer to $k$-uniform hypergraphs as $k$-graphs, use the notation $K^{(k)}_n$ to denote the complete $k$-graph on $n$ vertices and let $\bfg^{(k)}(n,p)$ denote the binomial random $k$-graph obtained by keeping each edge of $K_n^{(k)}$ independently with probability $p=p(n)$.  

Firstly, it is far from clear how to generalise the canonical Ramsey property (Definition~\ref{def:canonical ramsey property}) to $k$-graphs. This was done already by Erd\H{o}s and Rado \cite{ER50}. To start with, we make the following definition. 

\begin{definition}[Projection maps] \label{def:set projection}
    Let $2\leq k\in \NN$ and let $V$ be a set of size $v\geq k$
    and  $\sigma$ an ordering of $V$. Then for a (possibly empty)
    set $S\subseteq [k]$, we define the \emph{$S$-projection map}
    $\pi_S$ (with respect to $\sigma$) to be a function 
    $\pi_S:\binom{V}{k}\rightarrow\binom{V}{|S|}$ such that
    \[\pi_S(T)=\{v\in T:\text { $v$ is the $i$-th element of $T$ in the
      ordering $\sigma$ for some $i\in S$\}}.\]
\end{definition}

In words, the map $\pi_S$ pulls out the elements of an ordered $k$-set  that occupy the positions dictated by $S$. Note that $\pi_\emptyset(T)=\emptyset$  whilst $\pi_{[k]}(T)=T$ for all $T\in \binom{V}{k}$.  We also need the following definition. 

\begin{definition}[The reversal involution]
    For any set $S\subseteq [k]$, let \[\iota(S)=\{k-x+1:x\in S\}\] be the \emph{reverse} of~$S$. 
\end{definition}

The map $\iota:2^{[k]}\rightarrow 2^{[k]}$ partitions $2^{[k]}$ into  blocks $\{S,\iota(S)\}$ (some of which are pairs of sets and others are singletons).  
We say a collection $\cT\subseteq 2^{[k]}$ is a \textit{transversal} for the reversal involution if $\cT$ contains one set from each block. 
We can now describe the collection of canonical colourings of a hypergraph $H$.

\begin{definition}[Canonical copies]
  \label{def:canon colourings}
  Let $2\leq k\in \NN$ and $\cT\subseteq 2^{[k]}$ a transversal for
  the reversal involution on $[k]$. Further, let $H$ be a $k$-graph,
  $\sigma$ an ordering of $V(H)$ and $\chi:E(H)\rightarrow \mathbb{N}$
  an edge colouring of $H$. We say $(H,\chi)$ is \emph{canonical} with
  respect to $\sigma$ and $\cT$ if there is $S\in \cT$ and an
  \emph{injective} assignment of colours
  $\phi:\binom{V(H)}{|S|}\rightarrow \NN$ such that for every
  $e\in E(H)$, we have that $\chi(e)=\phi(\pi_S(e))$, where $\pi_S$ is
  the $S$-projection map with respect to $\sigma$.  If~$\chi$ is clear
  from the context, we simply say that~$H$ is canonical with respect
  to~$\sigma$ and~$\cT$.
\end{definition}

The canonical copies given by Definition \ref{def:canon colourings}
give \emph{all} the colour patterns that we are interested 
in. Indeed, if a copy of $H$ is canonical because $S=\emptyset$, then
this copy is monochromatic and for $S=[k]$, this copy is rainbow. The
intermediate $S$ cover all the other colour patterns. For example,
when $k=2$, the set $S=\{1\}$ corresponds to a copy that is
lexicographic with respect to the ordering $\sigma$ (Definition
\ref{def:lex_copy}). 
In fact, in Definition~\ref{def:canonical ramsey property} (where $k=2$), we define canonical copies with respect to $\sigma$ with the convention of choosing $\cT=\{\emptyset, \{1\},\{1,2\}\}$. 
If the transversal $\cT$ rather contains  the set $S=\{2\}$, this  corresponds to a copy of
$H$ that is lexicographic with respect to the ordering on~$V(H)$ which is reverse to $\sigma$. Similarly, for larger~$k$, when considering
all orderings~$\sigma$ of~$V(H)$ and transversals~$\cT$, certain colour patterns
of~$H$ will appear multiple times.

As with the graph case, the canonical Ramsey theorem of Erd\H{o}s and Rado \cite{ER50} implies that for any $2\leq k\in \NN$ and any $k$-graph $H$, if $n$ is sufficiently large, then $K_n^{(k)}$  has the \textit{$H$-canonical Ramsey property}, that is, for any edge colouring of $E(K_n^{(k)})$, any ordering $\sigma$ of $V(H)$ and any transversal $\cT$ for the reverse involution on $[k]$, there is a canonical copy of $H$ with respect to $\sigma$ and $\cT$.

\vspace{2mm}

Similarly, in the study of Ramsey properties of random graphs the
situation is considerably more involved for hypergraphs of uniformity
greater than 2. Indeed, such properties are not fully understood to
this date. The analogue of Definition~\ref{def:m2 density} for $3\leq
k\in \NN$ is given by the \emph{maximal $k$-density} of $H$ defined by
\begin{equation} \label{eq:mkdenisty}
m_k(H):=\max\left\{\frac{e(F)-1}{v(F)-k}:F\subseteq H,\, v(F)>k
    \right\},\end{equation}
where we assume that~$H$ has at least two edges.  The generalisation
of the $1$-statement of Theorem \ref{RR-original} to hypergraphs, by
which we mean that $\bfg^{(k)}(n,p)$ a.a.s.\ has the $H$-Ramsey property (for any number  of colours $r\geq 2$) when $p=\omega(n^{-1/m_k(H)})$, has been established for all $H$. However, this  took much longer than the graph case, eventually being resolved  by Friedgut, R\"odl and Schacht~\cite{FRS10} and, independently,  by Conlon and Gowers~\cite{CG16}.  
A corresponding $0$-statement has also been proven for various hypergraphs, in particular for cliques~\cite{NPSS17,T13}. 
The natural expectation that~$n^{-1/m_k(H)}$ should indeed be the
threshold for the $H$-Ramsey property in~$\bfg^{(k)}(n,p)$ (as in
Theorem~\ref{RR-original}), except for some simple cases, was
shattered by Gugelmann, Nenadov, Person, \v{S}kori\'c and Steger~\cite{GNPSST17}, who presented a family of exceptions richer than
in the graph case.  Nevertheless, this expectation has
been proved to hold for `most'
$k$-graphs~$H$ by Bowtell, Hancock and Hyde~\cite{bowtell25:_proof_kohay}.

As with Theorem \ref{thm:main}, we show that when considering
colourings that are compatible with some prefixed lists, taking
$p=\omega(n^{-1/m_k(H)})$ suffices to find canonical
copies. The definition of $G\canarrowL H$ in hypergraphs is completely analogous to the graph case except that one now insists that we find canonical copies with respect to all orderings $\sigma$ of $V(H)$ \textit{and} all transversals $\cT$ of the reverse involution on $[k$].

\begin{theorem}\label{thm:main-hyper}
  Let $2\leq k\in \NN$ and $H$ be a $k$-graph with at least two
  edges. Further, let $1 \leq r\in \bbn$ and
  $\cL:E(K^{(k)}_n)\rightarrow \bbn^r$ be a list assignment of
  colours.  If $\bfg\sim\bfg^{(k)}(n,p)$ with
  $p = \omega\left(n^{-1/m_k(H)} \right)$, then $\bfg\canarrowL H$
  a.a.s.
\end{theorem}   

\subsection{Alternative definitions for the canonical Ramsey property}
\label{sec:alts}

The original canonical Ramsey theorem of Erd\H{o}s and Rado~\cite{ER50} was
stated using slightly different notions to those used here. Indeed, they
consider \textit{ordered} copies of $K^{(k)}_m$ and $K^{(k)}_n$ on vertex sets $[m]$ and $[n]$, respectively. They prove that
if~$n$ is large enough in terms of~$m$, then for any colouring $\chi:E(K_n^{(k)})\rightarrow \NN$, one can find an \emph{ordered
embedding}\footnote{By an ordered embedding, we mean that $\psi(i)<\psi(j)$ in $[n]$ if and only if $i<j$ in $[m]$.} $\psi$
of $K^{(k)}_m$ in $K^{(k)}_n$  which is \textit{canonical} with respect to the ordering on $V(K_m^{(k)})=[m]$. In more detail, they show that   
 there is some $S\subseteq [k]$ and an injective assignment of colours $\phi:\binom{[m]}{|S|}\rightarrow \NN$ such that for every $e\in E(K^{(k)}_m)$, we have that $\chi(\psi(e))=\phi(\pi_S(e))$, with $\pi_S$ being the $S$-projection map with respect to the standard ordering on $[m]$.

Using this, we can see that for any $H$ with $m=|V(H)|$, if $n$ is sufficiently large, then $K_n^{(k)}\canarrow H$, that is, $K_n^{(k)}$ has the $H$-canonical Ramsey property. Indeed, 
fix some edge colouring $\chi$ of $E(K_n^{(k)})$, some ordering $\sigma$ of $V(H)$ and a transversal $\cT\subset2^{[k]}$ for the reversal involution on $[k]$. 
Labelling the vertices of $K_n^{(k)}$ by $[n]$ arbitrarily, 
the theorem of Erd\H{o}s and Rado~\cite{ER50}  gives some ordered embedding $\psi:[m]\rightarrow [n]$, a set $S\subseteq [k]$ and an injective map $\phi:\binom{[m]}{|S|}\rightarrow \NN$ such that $\chi(\psi(e))=\phi(\pi_S(e))$ for all $e\in \binom{[m]}{k}$. Now if $S\in \cT$, then we can embed $H$  by mapping the $i^{\text{th}}$ vertex of $V(H)$ according to $\sigma$, to the vertex $\psi(i)\in [n]$ for $i\in [m]$. If, on the other hand $S\notin \cT$, then we must have $\iota(S)\in \cT$ and we can embed $H$ by mapping the $i^{\text{th}}$ vertex of $V(H)$ according to $\sigma$ to the vertex $\psi(m-i+1)\in [n]$. In either case, the resulting copy of $H$ is canonical with respect to $\sigma$ and $\cT$.

The argument above shows that our notion of $H$-canonical Ramsey
property makes sense, but the reader may wonder if there are
alternative ways to define such a property. Indeed, one alternative which has more in common with the original theorem of Erd\H{o}s and Rado~\cite{ER50} is to look at \textit{ordered copies} of $H$ and $K_n^{(k)}$ with a completely analogous notion of an ordered embedding of $H$ being canonical as in  the definition in the theorem of Erd\H{o}s and Rado~\cite{ER50} when $H=K_m^{(k)}$. We remark that it would certainly be possible to work with this definition and prove our results here in the context of this ordered notion of canonical Ramsey properties.

There are several reasons for our choice to use the alternative characterisation of the canonical Ramsey property as in Definition \ref{def:canonical ramsey property} (see also Definition \ref{def:canon colourings}). Firstly, working with ordered graphs and ordered embeddings diverges from the classical Ramsey setting in unordered (hyper-)graphs. Secondly (and more importantly) when dealing with $H$ that are not complete, the ordered notion of the canonical Ramsey property is in fact weaker. This is already apparent in the setting of graphs where the ordered notion will find an ordered copy of $H$ that is either monochromatic, rainbow, min-coloured (each edge inherits its colour from its smaller endpoint) or max-coloured (each edge inherits its colour from its larger endpoint). Now consider $H$ to be a triangle with a pendent edge. By taking an ordering $\sigma$ of the vertices of $H$ in which the (unique) vertex of degree 3  comes first, our notion in Definition \ref{def:canonical ramsey property} guarantees that if there is no copy of $H$ which is monochromatic or rainbow, then one can find a copy of $H$ which has all vertices incident to the degree 3 vertex in one colour and the remaining edge in a distinct colour. On the other hand, such a colouring is not forced when considering the ordered notion of the canonical Ramsey property. 
Indeed, if one  uses only that in the absence of  monochromatic and rainbow $H$, one finds min- or max- coloured copies of all orderings of $H$, then one can adversarially choose min- or max- for each order in such a way that the considered colouring is never output.

This is the principal reason for 
why we have opted to state our results here without looking for
ordered embeddings, so as to capture all sets of unavoidable colour
patterns for graphs, as in \cite{AJ05,JW04}. The definition in hypergraphs is slightly more cumbersome (having to vary over all transversals $\cT$ of the reverse involution on $[k]$) but also captures more colour patterns than the canonical Ramsey notion via ordered hypergraphs.

\subsection{Related work}
\label{sec:related}

Until very recently, to our knowledge, there have been no results on
canonical Ramsey properties in random graphs. However, simultaneously
and independently to our work here and in~\cite{AKMM-2}, Kam{\v{c}}ev
and Schacht~\cite{KS23, kamcev25:_canon} obtained a remarkable result, completely
resolving the problem for the case when~$H=K_m$. Indeed, they show
that for $4\leq m\in \mathbb{N}$ and $p=\omega(n^{-2/(m+1)})$,
a.a.s. $\bfg(n,p)\canarrow K_m$. As $m_2(K_m)=(m+1)/2$, this
establishes the threshold for the canonical Ramsey property with
respect to complete (2-uniform) graphs~$K_m$. In contrast to our proofs, which
appeal to the method of hypergraph containers, their proof relies on
the transference principle of Conlon and Gowers~\cite{CG16}. Their
methods also allow them to obtain partial results in the case where
$H$ is a \textit{strictly balanced} graph, that is when $m_2(H)$ is achieved
only by $H$ itself. For such a graph $H$ and for
$p=\omega(n^{-1/m_2(H)})$, they can prove the existence of
monochromatic, rainbow or \textit{some} lexicographic copies
of~$H$ in any colouring of $\bfg\sim\bfg(n,p)$. Intriguingly, their
results do not generalise to the hypergraph setting, where it seems
new ideas are necessary.  

\subsection{Organisation and conventions} As the setting of graphs
(Theorem~\ref{thm:main}) is notationally simpler than the general
result for $k$-graphs (Theorem~\ref{thm:main-hyper}), we first
restrict our attention to Theorem~\ref{thm:main} and then simply
discuss the proof of Theorem~\ref{thm:main-hyper}, which is almost
identical, in Section~\ref{sec:hyperproof}. Also, for graphs, we
actually prove a stronger statement,
Theorem~\ref{thm:main_list_general}, which we need in our application
in~\cite{AKMM-2} and from which Theorem~\ref{thm:main} will
follow. Before proving Theorem \ref{thm:main_list_general} in
Section~\ref{sec:fullproof}, we provide some necessary tools in
Section~\ref{sec:prelims} and an overview of our proof in
Section~\ref{sec:proof}.

As usual, we omit floor and ceiling symbols whenever they are not
essential.  Finally, note that, in our main results, we may assume
without loss of generality that our target graph~$H$ has no isolated
vertices.  Assuming this will occasionally simplify our exposition.

\section{Preliminaries} \label{sec:prelims} 

In this section, we collect the necessary theory and tools needed in
our proof.  We will introduce the method of containers in
Section~\ref{sec:containers} and the theory of locally dense graphs in
Section~\ref{sec:localdense}. Before all of this though, we collect
the relevant notation.

For simplicity, given a graph $H$ we use $v(H)$ and $e(H)$,
respectively, for $|V(H)|$ and~$|E(H)|$.  For a $k$-uniform hypergraph
$\cH$ and for $U\subseteq V(\cH)$, we let $\cH[U]$ denote the
hypergraph induced by $\cH$ on $U$. Furthermore, for any vertex subset
$T\subset V(\cH)$, let $d_\cH(T)$ denote the number of edges of $\cH$
containing $T$ and, for $ 0\leq j\leq k$, let
$\Delta_j(\cH):=\max\{d_\cH(T): T\subset V(\cH), |T|=j\}$ denote the
maximum degree of a vertex set of size $j$ in $\cH$.

The binomial random graph $\bfg(n,p)$ refers to the probability
distribution of graphs on vertex set $[n]$ obtained by taking every
possible edge independently with probability $p=p(n)$. We say an event
happens asymptotically almost surely (a.a.s.\ for short) in $\bfg \sim
\bfg(n,p)$ if the probability that it happens tends to 1 as $n$ tends
to infinity. We will also use standard asymptotic notation throughout,
with asymptotics always being taken as the number of vertices $n$
tends to infinity. Finally, we use the notation $a=b\pm c$ to denote a
number $a$ between $b-c$ and $b+c$ and we omit floors and ceilings
throughout, so as not to clutter the arguments.

\subsection{The method of containers}
\label{sec:containers}

We will appeal to the method of hypergraph containers, developed by
Balogh, Morris and Samotij~\cite{BMS14}, and independently, by Saxton
and Thomason~\cite{ST15}. The key idea underlying this method is that
if a uniform hypergraph has an edge set that is evenly distributed,
then one can group the independent sets of the hypergraph into a
well-behaved collection of \emph{containers}. In more detail, these
containers are vertex subsets that are almost independent (in that
they induce few edges of the hypergraph), every independent set of the
hypergraph lies in some container and, crucially, we have a bound on
the number of containers. As there are far fewer containers than
independent sets in the hypergraph, reasoning about containers rather
than independent sets leads to more efficient arguments and this
technique has proven to be extremely powerful. Indeed, the setting of
independent sets in hypergraphs can be used to encode a wide range of
problems in combinatorics and the method of hypergraph containers has
been successfully exploited in a multitude of different
settings. Particularly relevant to our work here are the applications
of the method in sparse Ramsey theory, a program which was initiated
by Nenadov and Steger~\cite{NS16}, who reproved the 1-statement of Theorem \ref{RR-original} using containers.

Below, we state the container lemma in the form given in \cite[Theorem
2.2]{BMS14}. Before doing so, we need to establish some terminology
and definitions.

\begin{definition}[$(\cH,\eps)$-abundant set families]
  \label{def:H-abund}
  Let $\cH=(V,E)$ be a hypergraph and let $0<\eps\leq 1$. We say a
  family $\calf\subseteq 2^V$ is \emph{$(\calh,\eps)$-abundant} if the
  following hold:
  \begin{enumerate}
  \item \label{cond:1} $\calf$ is \emph{increasing}: for all
    $A,B\subseteq V$ with $A\in \calf$ and $A\subset B$, we have that
    $B\in \calf$;
  \item \label{cond:2} $\calf$ contains only large vertex sets: for
    all $A\in \cF$, we have that $|A|\geq \eps v(\cH)$;
  \item \label{cond:3} $\cH$ is \emph{$(\cF,\eps)$-dense}: for all
    $A\in \cF$, we have that $e(\cH[A])\geq \eps e(\cH)$.
  \end{enumerate}
\end{definition}

For a hypergraph $\cH=(V,E)$, we also define $\cali(\cH)\subset 2^V$
to be the collection of independent vertex sets in $\cH$. We now state
the container theorem \cite[Theorem 2.2]{BMS14} and we refer to the
discussion in \cite{BMS14} for motivation and context.

\begin{theorem}[Hypergraph Container Theorem]\label{thm:container-BMS}
  For every $k\in \bbn$ and $\eps,D_0>0$, there exists $D>0$ such that
  the following holds for $k$-uniform hypergraphs
  $\cH$. If~$\cF\subseteq 2^{V(\cH)}$ is an $(\cH, \eps)$-abundant set
  family and $q\in (0,1)$ is such that for each $j\in [k]$ we have
  $\Delta_{j}(\calh) \leq D_0 q^{j-1} e(\calh)/v(\calh)$, then there
  exists a family $\cals\subset 2^{V(\cH)}$ of `fingerprints' and two
  functions $f : \cals \to 2^{V(\cH)}\setminus \cF$ and $g :
  \cali(\calh) \to \cals$ such that:

\begin{enumerate}[(a)]
\item \label{item: not many containers} (`small' fingerprints)
  for each $S\in \cS$, we have that $|S|\leq Dqv(\cH)$;
\item \label{item: containment}(containment) for each $I\in
  \cali(\calh)$, we have that $g(I) \subset I$ and $I\setminus g(I)
  \subseteq f(g(I))$.
\end{enumerate}
\end{theorem}

In applications of Theorem \ref{thm:container-BMS}, the set
$\cC:=\{f(S)\cup S:S\in \cS\}$ is usually referred to as the set of
\emph{containers} for the hypergraph $\cH$. Note that property
\ref{item: not many containers} can be used to bound the size of $\cC$
whilst property \ref{item: containment} shows that for every
independent set $I\in \cI(\cH)$ there is some $C\in \cC$ such that
$I\subseteq C$.

\subsection{Locally dense graphs} \label{sec:localdense}

We will establish our main theorem in the context of random
sparsifications of locally dense graphs, a more general setting than
$\bfg(n,p)$ which will be useful for applications. Here we collect
some key properties of locally dense graphs.

\begin{definition}[$(\rho,d$)-dense graphs]
  For $\rho,d\in (0,1]$, we say a graph $\Gamma$ is
  \emph{$(\rho,d$)-dense} if the following holds: for every $S\subset
  V(\Gamma)$ with $|S|\geq \rho n$, the induced graph $\Gamma[S]$ has
  at least $d\binom{|S|}{2}$ edges.
\end{definition}

\begin{remark} \label{rem:simple local dense} In~\cite{RR95}, the
  authors observed that in order to establish that a graph $\Gamma$
  is $(\rho,d)$-dense it suffices to deal with subsets $S\subset
  V(\Gamma)$ of cardinality exactly $\rho n$.
\end{remark}

Graphs with the $(\rho,d)$-denseness property for $\rho = o(1)$ are
often called {\em locally dense} graphs and this property can be
viewed as a weak quasirandomness property. The following result
appears in~\cite[Lemma~2]{RR95} and can be proven by induction.

\begin{proposition}\label{prop:KNRS-conjecture for cliques}
  For every $m\geq 2$ and $d >0$ there exist $\rho, c_0>0$ such
  that if $\Gamma$ is an $n$-vertex $(\rho,d)$-dense graph and $n$ is
  sufficiently large, then $\Gamma$ contains at least $c_0n^m$ copies
  of~$K_m$.
\end{proposition}

We close this section with the following result about robustness of
the locally denseness property in the sense that it is
preserved after removing a small fraction of edges of a dense graph.

\begin{proposition}\label{prop:locally-dense_resilience}
  Suppose $\rho, d, \gamma>0$. If $\Gamma$ is a $(\rho,d)$-dense graph
  on $n$ vertices with $n$ sufficiently large, then every spanning
  subgraph $\Gamma'$ of $\Gamma$ with $e(\Gamma') \geq
  (1-\gamma)e(\Gamma)$ is a $(\rho,d')$-dense graph, with $d' := d -
  (2\gamma/\rho^2)$. In particular, if $\gamma \leq \rho^2 d/4 $, then
  such $\Gamma'$ is a $(\rho,d/2)$-dense graph.
\end{proposition}

\begin{proof}
  Let $\Gamma$ be a $(\rho,d)$-dense graph on $n$ vertices and let
  $\Gamma'$ be a spanning subgraph of $\Gamma$ with $e(\Gamma') \geq
  (1-\gamma) e(\Gamma)$ edges. If $S\subset V(\Gamma)$ has cardinality
  $\rho n$, then
  \begin{linenomath}
    \begin{align*}
      e_{\Gamma'}(S) \geq e_{\Gamma}(S)
      - \gamma e(\Gamma) \geq d\binom{|S|}{2} - \frac{\gamma}{2} n^2 =
      d\binom{|S|}{2} - \frac{\gamma}{\rho^2} \frac{|S|^2}{2} \geq d'
      \, \binom{|S|}{2},
    \end{align*}
  \end{linenomath}
  using that $n$ is sufficiently large in the last step and hence
  $|S|^2/2\leq 2 \binom{|S|}{2}$. Appealing to Remark~\ref{rem:simple
    local dense}, we have that $\Gamma'$ is a $(\rho,d')$-dense graph
  on $n$ vertices, as desired.
\end{proof}

\section{Outline of the proof}
\label{sec:proof}
 
Our proof follows the scheme of Nenadov and Steger~\cite{NS16} and
appeals to the method of hypergraph containers (see
Section~\ref{sec:containers}). We start by giving a rough outline of
the proof.

Suppose we want to find a monochromatic $H$ in any $r$-colouring of
$\bfg\sim \bfg(n,p)$. Our idea is to create an auxiliary hypergraph
$\cH$ whose vertex set consists of $r$ copies of $E(K_n)$ (one copy
for each colour) and whose edge set encodes the monochromatic copies
of~$H$ in~$K_n$.  Thus, an edge in~$\cH$ corresponds to a copy of~$H$
entirely contained in one of the~$r$ copies of~$K_n$; putting it
another way, $\cH$~is the disjoint union of~$r$ copies of the
$|E(H)|$-uniform hypergraph on~$E(K_n)$ that encodes all the
copies of~$H$ in~$K_n$.  The key observation is the following: any
colouring of~$\bfg$ that avoids monochromatic copies of~$H$ can be
identified with an independent set in~$\cH$.  Using containers, one can
efficiently group together these independent sets and identify a small
set~$\cC$ of \emph{containers} such that each independent set of $\cH$
belongs to one such container. The proof then proceeds by showing
that, for each \emph{fixed} container $C\in \cC$, it is very
unlikely\footnote{That is, it happens with probability
  $\exp(-\Omega(n^2p))$.} that the graph $\bfg$ lies within $C$. By
this, we mean that it is unlikely that there is a colouring of $\bfg$
that, when mapped in the obvious way to a vertex subset of $\cH$,
corresponds to a subset of $C$. The proof follows by performing a
union bound over the choices for a container $C\in\cC$.

This containers approach relies crucially on the fact that the
hypergraph $\cH$ that encodes the monochromatic copies of $H$ has
$v(\cH)=O(n^2)$ and so satisfies the the degree constraints of
Theorem~\ref{thm:container-BMS} with $q=\Theta(n^{-1/m_2(H)})$. Having
an unbounded number of copies of $E(K_n)$ in~$V(\cH)$, which
corresponds to an unbounded number of colours available, leads to an
adjustment on the degree constraints and the parameter $q$ in
Theorem~\ref{thm:container-BMS}. This renders the upper bound on the
number of containers useless and the proof falls apart. In our proof
we avoid such a problem by creating a hypergraph whose vertex set is
composed by~$r$ copies of~$E(K_n)$ with each copy of an edge
corresponding to a choice of colour of that edge according to the list
of available colours.  The edge set of the hypergraph then encodes the
canonical copies of $H$. The fact that each list is bounded allows us
to apply Theorem~\ref{thm:container-BMS} to our hypergraph with the
correct parameters.  As previously discussed, we can in fact prove the
following more applicable result, which is a strengthening of
Theorem~\ref{thm:main} and applies to random sparsifications of
locally dense graphs.

\begin{theorem}[Sparse canonical Ramsey theorem with list
  constraints]
  \label{thm:main_list_general}
  Let $H$ be a graph with at least two edges and suppose
  $1 \leq r\in \bbn$ and $d>0$.  Then there exist $\rho,\,c >0$ such
  that the following holds.  Let~$\sigma$ be an ordering of $V(H)$,
  let $\Gamma$ be an $n$-vertex $(\rho,d)$-dense graph and let
  $\call:E(\Gamma)\rightarrow \mathbb{N}^r$ be a list assignment.  If
  $p = \omega(n^{-1/m_2(H)})$ and~$n$ is large enough, then with
  probability at least $1- \exp{(-cp n^2)}$, any edge colouring $\chi$
  of $\Gamma \cap \bfg(n,p)$ that is compatible with $\cL$ contains a
  canonical copy of $H$ with respect to $\sigma$.
\end{theorem}   

Theorem \ref{thm:main} follows from
Theorem~\ref{thm:main_list_general} applied to $\Gamma=K_n$, which is
$(\rho,1)$-dense for all~$\rho>0$, and a union bound over the $v(H)!$
orderings $\sigma$ of $V(H)$.

\section{Proof}
\label{sec:fullproof}
In this section we make the outline given in Section~\ref{sec:proof}
precise, proving Theorem~\ref{thm:main_list_general}.

\subsection{A hypergraph encoding canonical copies}
\label{sec:hypergraph encoding}
With an eye to apply the hypergraph container theorem
(Theorem~\ref{thm:container-BMS}), we first define an appropriate hypergraph
that encodes the canonical copies of $H$ with respect to a colouring
that is compatible with some list assignment. We make the following
definition.

\begin{definition}[Canonical copy hypergraphs]
  \label{def:canon hypergraph}
  Given a graph $H$, an ordering $\sigma$ of $V(H)$, an integer $r\geq
  1$, an $n$-vertex graph $\Gamma$ and a list assignment
  $\cL:E(\Gamma)\rightarrow \bbn ^r$, we define the \emph{canonical
    copy hypergraph} $\cH=\cH^{\sigma}_H(\Gamma,\cL)$ as follows:

  The hypergraph $\cH$ is $e(H)$-uniform with $V(\cH)=E(\Gamma)\times
  [r]$. A collection $\{(e_i,s_i):1\leq i\leq e(H)\}\subseteq V(\cH)$
  of vertices form an edge of $\cH$ if and only if the collection
  $\{e_i:1\leq i\leq e(H)\}\subseteq E(\Gamma)$ form a copy of $H$ in
  $\Gamma$ that is canonical with respect to~$\sigma$ when each edge
  $e_i$ is coloured by the $s_i$-th colour $\cL(e_i)[s_i]$ of
  $\cL(e_i)$.
\end{definition}

Let~$\cH=\cH^{\sigma}_H(\Gamma,\cL)$ be a canonical copy hypergraph as
in Definition~\ref{def:canon hypergraph}.  For $W\subseteq V(\cH)$,
let the \emph{graph shadow} $G_W\subseteq\Gamma$ of~$W$ be the
subgraph of~$\Gamma$ spanned by
$E(G_W):=\{e\in \Gamma: (e,s)\in W \mbox{ for some } s\in [r]\}$.
Note that the vertex set~$W$ is a set of pairs in
$E(\Gamma)\times[r]$, and hence the graph shadow~$G_W$ is just the
graph obtained by projecting the vertices of~$W$ onto their first
coordinates.

We also need the following observation on the degrees of~$\cH$.

\begin{lemma}
  \label{lem:can copy hyp degrees}
  For any graph $H$, ordering $\sigma$ of $V(H)$, integer $r\geq 1$,
  graph $\Gamma$ with~$n$ vertices and list assignment
  $\cL:E(\Gamma)\rightarrow \bbn ^r$, the canonical copy hypergraph
  $\cH=\cH^{\sigma}_H(\Gamma,\cL)$ satisfies the following: for $1\leq
  j \leq e(H)$, we have $ \Delta_{j}(\calh) 
  \leq r^{e(H)} \left(n^{-1/m_2(H)}\right)^{j-1} n^{v(H)-2}.$
\end{lemma}

\begin{proof}
  Note that the result holds for $j=1$, as any edge of $H$ is in at
  most $r^{e(H)}n^{v(H)-2}$ canonical copies of~$H$ with respect
  to~$\sigma$. Then, we assume that $2\leq j \leq e(H)$ and fix an
  arbitrary $j$-element set $U=\{(e_i,s_i):1\leq i \leq j\}\subseteq
  V(\cH)$.  It suffices to show that the number $d_\cH(U)$ of edges of $\cH$
  containing $U$ is at most
  $r^{e(H)} \left(n^{-1/m_2(H)}\right)^{j-1} n^{v(H)-2}$.

  Consider the graph $F=G_U\subseteq\Gamma$ induced by the edge set
  $E(G_U) = \{e_i:1\leq i \leq j\}$. If $F$ is not a subgraph of some
  copy of $H$ in $\Gamma$, then $d_\cH(U)=0$ and we are done. If~$F$
  is a copy of a subgraph of $H$, then note that there are at most
  $n^{v(H)-v(F)}$ extensions of $F$ to a copy of $H$ in $K_n$ and at
  most $r^{e(H)-e(F)}$ choices of second coordinate for the vertices
  of~$\cH$ corresponding to the copy of $H$ that
  extends~$F$.  Therefore,
  \[d_\cH(U)\leq r^{e(H)-e(F)}n^{v(H)-v(F)} \leq
    r^{e(H)}n^{-(v(F)-2)}n^{v(H)-2}.\]
  By using the definition of
  $m_2(H)$, we conclude that
  \[
    d_\cH(U)\leq r^{e(H)}\left(n^{-\tfrac{v(F)-2}{j-1}}\right)^{j-1}n^{v(H)-2}\leq
    r^{e(H)} \left(n^{-1/m_2(H)}\right)^{j-1} n^{v(H)-2},
  \]
  as required.
\end{proof}

\subsection{Supersaturation for canonical copies}
\label{sec:supersat}
 
Next we prove a supersaturation-type result for canonical copies in
locally dense graphs.

\begin{lemma}[Canonical supersaturation]
  \label{lem:supersaturation-type}
  Let $H$ be a graph and let $d\in (0,1]$.  Then there exist
  $\rho,\, c_0>0$ such that for any $n$-vertex $(\rho,d)$-dense graph
  $\Gamma$ with large enough~$n$, any ordering~$\sigma$ of~$V(H)$ and
  any colouring $\chi:E(\Gamma)\rightarrow \bbn$, one can find at
  least~$c_0 n^{v(H)}$ copies of~$H$ in $\Gamma$ coloured by~$\chi$ in
  a canonical fashion with respect to~$\sigma$.
\end{lemma}	
    
\begin{proof}
  Fix a graph $H$ and $d\in (0,1]$. It is well known by the classical
  canonical Ramsey theorem~\cite{ER50} that there is a
  positive integer $R$ such that $K_R \canarrow K_{v(H)}$. Now fix
  $\rho,c_0>0$ as output by Proposition~\ref{prop:KNRS-conjecture for
    cliques} with input $m=R$ and $d>0$. Then fix $\Gamma$, an
  ordering $\sigma$ of $V(H)$ and $\chi$ as in the statement.
 
  Note that in a canonical copy of $K_{v(H)}$ (with respect to any
  ordering), one can find a canonical copy of $H$ with respect to
  $\sigma$. By Proposition~\ref{prop:KNRS-conjecture for cliques}, the
  graph $\Gamma$ contains at least $c_0n^R$ copies of $K_R$ and, as
  $K_R \canarrow K_{v(H)}$, each of these copies contains a canonical
  copy of $K_{v(H)}$ and hence a canonical copy of $H$ with respect to
  $\sigma$. On the other hand, each copy of $H$ in $\Gamma$ is
  contained in at most $n^{R-v(H)}$ copies of $K_R$ (in
  $\Gamma$). Hence, the number of canonical copies of $H$ in $\Gamma$
  coloured by $\chi$ is at least $c_0 n^{R}/n^{R-v(H)} \geq
  c_0n^{v(H)}$.
\end{proof}

Using Lemma~\ref{lem:supersaturation-type}, we can show that the
family of sets of vertices of~$V(\cH)$ with large graph shadows are
abundant with respect to~$\cH$, in the sense of
Definition~\ref{def:H-abund}.

\begin{lemma}
  \label{lem:abund}
  For all $d>0$ and $1\leq r\in \bbn$, there exist
  $\rho,\,\gamma,\,\eps>0$ such that the following holds.

  Let~$\Gamma$ and~$H$ be graphs such that~$\Gamma$ is
  $(\rho,d)$-dense and $e(H)\geq1$, let~$\sigma$ be an ordering
  of~$V(H)$ and let $\cL:E(\Gamma)\rightarrow \bbn ^r$ be given.
  Furthermore, let $\cH=\cH^{\sigma}_H(\Gamma,\cL)$ and define
  $\cF:=\{W\subseteq V(\cH):e(G_W)\geq (1-\gamma)e(\Gamma)\}\subseteq
  2^{V(\cH)}$.  Suppose $n=v(\Gamma)$ is large enough.  Then~$\cF$ is
  $(\cH,\eps)$-abundant.
\end{lemma}

\begin{proof}
  Fix $\rho,\,c_0>0$ as output by Lemma~\ref{lem:supersaturation-type}
  with input $d/2$ rather than $d$. Furthermore, fix $0<\gamma\leq
  \rho^2d/4$ and $0<\eps<c_0/(2r^{e(H)})$. We aim to show that $\cF$
  is $(\cH,\eps)$-abundant and so we need to establish conditions
  \eqref{cond:1}--\eqref{cond:3} of Definition \ref{def:H-abund}. Note
  that \eqref{cond:1} is immediate from the definition of $\cF$ and
  condition \eqref{cond:2} follows from our definition of $\eps$ and
  the fact that any $W\in \cF$ must have $|W|\geq
  (1-\gamma)e(\Gamma)\geq e(\Gamma)/2=v(\cH)/(2r)$.

  Therefore it remains to prove condition \eqref{cond:3}. For this,
  fix $W\in \cF$ and consider the graph shadow $G=G_W$. Moreover,
  define a colouring $\chi:E(G)\rightarrow \bbn$ so that for each edge
  $e\in E(G)$ we have that $\chi(e)=\cL(e)[s]$ for some $s$ with
  $(e,s)\in W$. Now note that each copy of $H$ in $G$ coloured by
  $\chi$, which is canonical with respect to $\sigma$, corresponds to
  an edge in~$\cH[W]$.  Therefore it suffices to prove the existence
  of~$\eps e(\cH)$ canonical copies of~$H$ in~$G$ coloured by~$\chi$.
  Let~$G'$ be the spanning subgraph of~$\Gamma$ with edge set~$E(G)$
  (that is, add to~$G$ the vertices in $V(\Gamma)\setminus V(G)$ as
  isolated vertices).  By the choice of~$\gamma$ and
  Proposition~\ref{prop:locally-dense_resilience}, the graph~$G'$ is
  $(\rho,d/2)$-dense and so, by Lemma~\ref{lem:supersaturation-type},
  contains at least~$c_0n^{v(H)}$ canonical copies of~$H$ with respect
  to~$\sigma$.  Our assumption that~$H$ has no isolated vertices
  implies that these copies are subgraphs of~$G$.  The lemma thus
  follows from our definition of~$\eps$ and the fact that
  $e(\cH)\leq n^{v(H)}r^{e(H)}$.
\end{proof}

\subsection{Proof of Theorem \ref{thm:main_list_general}}

We are now in a position to prove our main result.

\begin{proof}[Proof of Theorem \ref{thm:main_list_general}]
  Let $H$, $1\leq r\in \bbn$ and $d>0$ be as in the statement and let
  $\rho,\gamma,\eps, c_0>0$ be small enough for us to be able to apply
  Lemmas~\ref{lem:supersaturation-type} and~\ref{lem:abund}. Further,
  fix $0<c<\gamma d/16$, $D_0=r^{e(H)+1}/c_0$, and let $D>0$ be as
  output by Theorem~\ref{thm:container-BMS} applied with $k=e(H)$.
  
  Next we fix an ordering~$\sigma$ of~$V(H)$, some $n$-vertex
  $(\rho,d)$-dense graph~$\Gamma$ with~$n$ large enough and some list
  assignment $\cL:E(\Gamma)\rightarrow \bbn^r$.  Let
  $\cH=\cH^{\sigma}_H(\Gamma,\cL)$ be the canonical copy hypergraph as
  in Definition \ref{def:canon hypergraph} and $\cF\subseteq
  2^{V(\cH)}$ be as in Lemma~\ref{lem:abund}. We claim that
  all the conditions of the Hypergraph Container Theorem
  (Theorem~\ref{thm:container-BMS}) are satisfied 
  with $q:=n^{-1/m_2(H)}$. Indeed the condition that $\cF$ is
  $(\cH,\eps)$-abundant is given by Lemma~\ref{lem:abund} and the
  degree 
  conditions on $\cH$ follow from Lemma \ref{lem:can copy hyp degrees}
  and the facts that $v(\cH)=r\binom{n}{2}\leq rn^2$ and $e(\cH)\geq
  c_0 n^{v(H)}$, as can be seen by applying
  Lemma~\ref{lem:supersaturation-type} to the colouring defined by
  each edge taking, say, the first colour in its list.
    
  By Theorem~\ref{thm:container-BMS} we thus get a collection of
  fingerprints $\cS\subseteq 2^{V(\cH)}$ such that $|S|\leq Dqv(\cH)$
  for all $S\in \cS$. Moreover, there are two functions $f : \cals \to
  2^{V(\cH)}\setminus \cF$ and $g : \cali(\calh) \to \cals$ such that
  for each $I\in \cali(\calh)$, we have that $g(I) \subset I$ and
  $I\setminus g(I) \subseteq f(g(I))$. Now for each $S\in \cS$, define
  $C(S):= f(S)\cup S \subseteq V(\cH)$. We claim that for all $S\in
  \cS$, we have that
  \begin{equation}
    \label{eq:missing edges}
    e\big(\Gamma\setminus G_{C(S)}\big)\geq
    \frac{\gamma}{2}e(\Gamma)\geq \frac{\gamma d}{8}n^2.
  \end{equation}
  Indeed this follows because $f(S)\notin \cF$ and so
  $e(G_{f(S)})<(1-\gamma)e(\Gamma)$ and as $n$ is sufficiently large,
  $e(G_S)\leq |S|\leq Dqv(\cH)\leq \gamma e(\Gamma)/2$.

  For an $n$-vertex graph $G\subseteq \Gamma$ and a colouring $\chi
  :E(G)\rightarrow \bbn$ compatible with $\cL$, let
  $W(G,\chi)\subseteq V(\cH)$ be the set of vertices $(e,s)\in V(\cH)$
  such that $e\in E(G)$ and $\chi(e)=\cL(e)[s]$. Note that if the
  graph $G$, when coloured by $\chi$, contains no canonical copies of
  $H$ with respect to $\sigma$, then $W(G,\chi)\in \cI(\cH)$, that is,
  $W(G,\chi)$ is an independent set. Therefore, by property
  \ref{item: containment} of Theorem \ref{thm:container-BMS}, the
  event that $\bfg\sim \bfg(n,p)$ is such that there is a colouring
  of $\bfg\cap\Gamma$ compatible with~$\cL$ that avoids canonical
  copies of~$H$ is contained in the event 
  $\bigcup_{S\in \cS}\Phi(S)$, where
  \[
    \Phi(S):=\{\exists\cL\mbox{-compatible
    }\chi:E(\bfg\cap\Gamma)\rightarrow \bbn \mbox{ such that } S\subseteq
    W(\bfg,\chi)\subseteq C(S)\}.
  \]
Now note that if~$\Phi(S)$ occurs then $G_S\subseteq \bfg$ and
$\bfg\setminus G_S\subseteq G_{C(S)}$, and these events are
independent.  Hence for a fixed $S\in \cS$, we have  
\begin{linenomath}
  \begin{align*}
    \Pr[\Phi(S)]&\leq \Pr[G_S\subseteq \bfg]\cdot \Pr[\bfg\setminus G_S\subseteq G_{C(S)}]\\ 
                &\leq p^{e(G_S)} \cdot \Pr[\bfg\cap (\Gamma\setminus G_{C(S)})=\emptyset]\\ 
                & \leq p^{e(G_S)} \cdot (1-p)^{\gamma d n^2/8}\\
                & \leq p^{e(G_S)} e^{-\gamma d n^2p/8},
\end{align*}
\end{linenomath}
using \eqref{eq:missing edges} in the penultimate inequality.
Therefore, the probability that there is an $\cL$-compatible colouring
of $\bfg\cap\Gamma$ avoiding canonical copies of~$H$ with respect
to~$\sigma$ is at most
\begin{equation}
  \label{eq:0}
  \sum_{S\in\cS} \Pr[\Phi(S)]\leq e^{-\gamma d
    n^2p/8}\sum_{S\in\cS} p^{e(G_S)}.
\end{equation}
For a given $0\leq t\leq e(\Gamma)$, the number of
sets~$S\subset V(\cH)$ such that~$e(G_S)=t$ is at most
${e(\Gamma)\choose t}(2^r-1)^t\leq{n^2\choose t}2^{rt}$.
Theorem~\ref{thm:container-BMS}\ref{item: not many containers} tells 
us that $e(G_S)\leq T$ for every $S\in\cS$, where $T=rDqn^2$.  Hence,
the summation on the right-hand side of~\eqref{eq:0} is at most
\begin{equation}
  \label{eq:1}
  \sum_{0\leq t\leq T}{n^2\choose t}2^{rt}p^t
  \leq2{n^2\choose T}2^{rT}p^T\leq2\left(e2^rn^2p\over T\right)^T,
\end{equation}
where the first inequality follows by comparing the sum with a
geometric series, assuming that $p\geq2^{1-r}rDq$.  Thus, the
right-hand side of~\eqref{eq:0} is at most
\begin{linenomath}
  \begin{align*}
    2e^{-\gamma dn^2p/8}\left(e2^rn^2p\over T\right)^T
    &\leq2e^{-\gamma dn^2p/8}\left(e2^rn^2p\over rDqn^2\right)^{rDqn^2}\\
    &\leq2\exp\left\{n^2p\left(
      -{1\over8}\gamma d+rD{q\over p}\log{e2^rp\over rD q}
      \right)\right\}\\
    &\leq e^{-\gamma dn^2p/10},
\end{align*}
\end{linenomath}
as long as~$p/q$ is larger than some constant that depends only
on~$r$, $D$, $\gamma$ and~$d$.  This completes the proof.
\end{proof}

\section{Hypergraphs} \label{sec:hyperproof}
In this section, we discuss the proof of Theorem~\ref{thm:main-hyper}, which is very similar to our proof of Theorem \ref{thm:main} with essentially only notational differences. Indeed, it is  well known that one of the key advantages of the method of hypergraph containers is that proofs in the graph case easily generalise to hypergraphs. In any case, for completeness, we discuss some of the details. 

Firstly, note that we concentrate solely on the random subgraph
$\bfg^{(k)}(n,p)$ and so do not need to worry about a locally dense
host $\Gamma$ as in the proof of Theorem~\ref{thm:main_list_general}
(see the discussion at the end of this section). This leads to some
simplifications in the proof. As in Definition \ref{def:canon
  hypergraph}, for the hypergraph $H$ in question, an ordering
$\sigma$ of $V(H)$, an integer $r \geq 1$ and a list assignment
$\cL:E(K_n^{(k)})\rightarrow \NN^r$, we define an auxiliary hypergraph
$\cH=\cH_H^\sigma(\cL)$ on vertex set $V(\cH)=E(K_n^{(k)})\times [r]$
whose edges encode canonical copies (Definition \ref{def:canon
  colourings}) of $H$ with respect to $\sigma$, when choosing colours
from the lists. As in Lemma 4.2, it follows from the
definition~\eqref{eq:mkdenisty} of $m_k(H)$ that
\begin{equation} \label{eq:hyp degrees}
  \Delta_j(\cH)\leq r^{e(H)}(n^{-1/m_{k}(H)})^{j-1}n^{v(H)-k},
\end{equation}
for all $1\leq j\leq e(H)$. Also as in the proof for graphs, for a
subset $W\subseteq V(\cH)=E(K_n^{(k)})\times [r]$, we define the
\emph{hypergraph shadow} $G_W$ to be the subhypergraph
of $K_n^{(k)}$ obtained by taking the hyperedges featured in $W$. We
further define $\cF=\cF(\gamma):=\{W\subseteq V(\cH):e(G_W)\geq
(1-\gamma)\binom{n}{k}\}$ to be the vertex subsets of $V(\cH)$ with
large hypergraph shadow.

We then fix some small $\gamma,\eps>0$ and need to show that
$\cF=\cF(\gamma)$ is $(\cH,\eps)$-abundant (Definition
\ref{def:H-abund}). Conditions \eqref{cond:1} and \eqref{cond:2}
follow directly from the definition of $\cF$ and for~\eqref{cond:3},
we need to prove supersaturation of canonical copies. As in Lemma
\ref{lem:supersaturation-type}, this can be done by averaging. Indeed,
by the canonical Ramsey theorem of Erd\H{o}s and Rado \cite{ER50},
there is some $R\in \NN$ such that any colouring of $K_R^{(k)}$
results in a canonical copy of $K^ {(k)}_{v(H)}$ and hence a canonical
copy of $H$ (with respect to $\sigma$). On the other hand, every copy
of~$H$ in $K_n^{(k)}$ is contained in at most $n^{R-v(H)}$ copies of
$K_{R}^{(k)}$. Therefore, as in Lemma~\ref{lem:supersaturation-type},
any colouring of~$K_n^{(k)}$ results in at least~$c_0n^{v(H)}$
canonical copies of~$H$ with respect to~$\sigma$, for some $c_0>0$.
This in turn implies that~$\cH$ is $(\cF,\eps)$-dense
(recall~\eqref{cond:3}), as given some $A\in \cF$, one can colour the
edges $e$ of $G_A$ with a colour in $\cL(e)$ corresponding to a vertex
$(e,s)\in A$ and colour edges of $K_n^{(k)}\setminus G_A$ arbitrarily.
This gives rise to at least~$c_0n^{v(H)}$ canonical copies of~$H$ and
by deleting any such copy of~$H$ containing an edge of
$K_n^{(k)}\setminus G_A$, we destroy at most~$\gamma n^{v(H)}$ of
these copies.  Thus, as~$\gamma$ is taken to be much smaller
than~$c_0$, we get at least $c_0n^{v(H)}/2\geq \eps e(\cH) $ canonical
copies of~$H$, each of which corresponds to an edge in~$\cH[A]$.

 We are therefore in a position to apply Theorem \ref{thm:container-BMS} to $\cH$ and $\cF$ with $q=n^{-1/m_k(H)}$, also considering \eqref{eq:hyp degrees}. The rest of the proof follows analogously to the proof of Theorem~\ref{thm:main_list_general}, by performing a union bound over the containers $C(S):=f(S)\cup S\subseteq V(\cH)$ output by Theorem \ref{thm:container-BMS} (as well as a union bound over orderings $\sigma$ of $V(H)$). As in \eqref{eq:missing edges}, we crucially use that if $\bfg\sim \bfg^{(k)}(n,p)$ can be coloured in a canonical $H$-free way that corresponds to an independent set lying in some fixed container $C(S)$, then as $f(S)\notin \cF$, there is some set of $\Omega(n^k)$ edges of $K_n^{(k)}$ that are not hit by $\bfg$, an event that occurs with probability less than $e^{-\Omega(n^kp)}$.

 We close by remarking that a  generalisation of Theorem \ref{thm:main_list_general} to hypergraphs would be more complicated than the version (Theorem \ref{thm:main-hyper}) we give here. Indeed, the straightforward generalisation of a locally dense $k$-graph $\Gamma$, namely that sets $S\subseteq V(\Gamma)$ of size $o(n)$ have positive density, is not enough to guarantee the supersaturation of copies of cliques, or even a single copy of a $K_{k+1}^{(k)}$ in $\Gamma$ (as noted by R\"odl using a construction from \cite{EH72}). A much stronger pseudorandom condition would be needed for $\Gamma$ so that an analogue of Theorem \ref{thm:main_list_general} holds for hypergraphs of larger uniformity. Whilst this should certainly be possible, as we do not yet have applications for such a result, we do not pursue this here. 

\subsection*{Acknowledgements} We are grateful to the anonymous referees for their careful reading of the previous version of the manuscript and their helpful suggestions. 
 
\bibliographystyle{amsplain} 
\bibliography{bibliography-mrefed}

\end{document}